\documentclass{article}
\usepackage[english]{babel}
\usepackage{amsmath,amssymb,graphicx,bbm,latexsym,calc,capt-of,ifthen,theorem}

\newcommand{\longdownarrow}{{\mbox{\rotatebox[origin=c]{-90}{$\longrightarrow$}}}}
\newcommand{\tmem}[1]{{\em #1\/}}
\newcommand{\tmop}[1]{\ensuremath{\operatorname{#1}}}
\newcommand{\tmstrong}[1]{\textbf{#1}}
\newenvironment{proof}{\noindent\textbf{Proof\ }}{\hspace*{\fill}$\Box$\medskip}
\newtheorem{lemma}{Lemma}
{\theorembodyfont{\rmfamily}\newtheorem{question}{Question}}
\newtheorem{theorem}{Theorem}
\newcommand{\tmfloatcontents}{}
\newlength{\tmfloatwidth}
\newcommand{\tmfloat}[5]{
  \renewcommand{\tmfloatcontents}{#4}
  \setlength{\tmfloatwidth}{\widthof{\tmfloatcontents}+1in}
  \ifthenelse{\equal{#2}{small}}
    {\setlength{\tmfloatwidth}{0.45\linewidth}}
    {\setlength{\tmfloatwidth}{\linewidth}}
  \begin{minipage}[#1]{\tmfloatwidth}
    \begin{center}
      \tmfloatcontents
      \captionof{#3}{#5}
    \end{center}
  \end{minipage}}

\begin{document}

\title{Manifolds with the fixed point property and their squares}

\author{Slawomir Kwasik\footnote{The first author was supported by the Simons Foundation Grant 281810 2010 Mathematics Subject Classification. Primary 55M20, 55M35; Secondary 57R18}  and Fang Sun}

\maketitle

\begin{abstract}
  The Cartesian squares (powers) of manifolds with the fixed point property
  (f.p.p.) are considered. Examples of manifolds with the f.p.p. are
  constructed whose symmetric squares fail to have the f.p.p..
\end{abstract}

A topological space X has the fixed point property (f.p.p.) if for every
continuous map $f : X \rightarrow X$ there exists a fixed point, that is, a
point $x \in X$ such that $f (x) = x$. There are plenty of examples of (nice)
spaces which fail to have the f.p.p. and there are examples of spaces with the
f.p.p..

The celebrated Theorem of Brouwer (cf. [7]) asserts that the n-dimensional cube
$I^n$ has the f.p.p.. On the other hand the n-dimensional sphere $S^n$ fails
to have the f.p.p..

The especially important role in the Fixed Point Theory is played by the
Lefschetz Fixed Point Theorem (cf. [5]). To be more specific:

Let X be a nice space , say a compact ANR (this includes finite CW-complexs
and compact topological manifolds). Let $\Lambda$ be a field. A map $f : X
\rightarrow X$ induces a homomorphism (linear tranformation)
\[ f_{\ast i} : H_i (X, \Lambda) \rightarrow H_i (X, \Lambda), \quad i = 0, 1,
   2, \cdots \]
The Lefschetz number $L (f, \Lambda)$ of a map $f : X \rightarrow X$ is
defined as $L (f, \Lambda) = \underset{i}{\Sigma} (- 1)^i \tmop{tr} f_{\ast
i}$ where $\tmop{tr} f_{\ast i}$ is the trace of $f_{\ast i}$.

\begin{theorem}
  (Lefschetz Fixed Point Theorem). Let $f : X \rightarrow X$ be a map. If $L
  (f, \Lambda) \neq 0$ then $f$ has a fixed point. 
\end{theorem}

Now since $L (f, \mathbbm{Q}) = 1$ for every continuous $f : I^n \rightarrow
I^n$ and the field of rational numbers $\mathbbm{Q}$, then the theorem of
Brower is a very special case of the Lefschetz Fixed Point Theorem.

There are more direct proofs of the Brower Fixed Point Theorem, but all of
them are surprisingly demanding given its elementary formulation.

A basic calculus argument shows that the interval I has the f.p.p. and thus
the most tempting attempt of proving the theorem of Brouwer would be
mathematical induction.

This is in turn is directly related to the general question raised by
Kuratowski in 1930 (cf. [12]).

\begin{question}
  Suppose $X, Y$ are locally connected and compact spaces with the f.p.p.,
  does the Cartesian product $X \times Y$ have the f.p.p.? 
\end{question}

It turns out that the answer to the above question is NO. The case of
polyhedra was treated by W. Lopez in [13] and the construction of
corresponding example is far from simple.

More refined example of closed manifolds $\mathcal{M}, \mathcal{N}$ with the
f.p.p. such that $\mathcal{M} \times \mathcal{N}$ admits a fixed point free
map was provided by S. Husseini in [11]. The construction in [11] is quite
involved and the technical difficulties are very substantial. In particular
the crucial fact which makes the construction in [11] to work is that
$\mathcal{M} \neq \mathcal{N}$.

This led to the following question which is considered to be one of the most
important open problems in the classical Fixed Point Theory (cf. [6]).

\begin{question}
  Does there exist a closed manifold $\mathcal{M}$ with the f.p.p. such that
  its Cartesian square $\mathcal{M}^2 =\mathcal{M} \times \mathcal{M}$ fails
  to have the f.p.p.?
\end{question}

The main purpose of this note is to rekindle the interest in the above
question. Even though at present we are not able to answer this question we
show that in the presence of an additional symmetry the answer is positive.

Namely, let $X$ be a topological space. The quotient space $X (n) = X^n /
S_n$, where the symmetric group $S_n$ acts on $X^n = X \times \cdots \times X$
by coordinate permutation, is called the $n$-th symmetric product of $X$. In
particular the symmetric square $X (2)$ is given by $X (2) = (X \times X)
/\mathbbm{Z}_2$. The symmetric product plays an important role in the algebraic
and geometric topology (cf. [1], [3], [7], [8]) as well as in algebraic
geometry (cf. [1]).

If $\mathcal{M}$ is a $k$-dimensional closed smooth manifold, then for $k
\leqslant 2$, $\mathcal{M}(n)$ is a manifold (possibly with a boundary). For
$k > 2$, $\mathcal{M}(n)$ is not a manifold but it is a compact polyhedron.

Here are some examples (cf. [1], [14]):

For $\mathcal{M}=\mathbbm{R}P^2$, $\mathcal{M}(n) =\mathbbm{R}P^{2 n}$.

For $\mathcal{M}= S^2$, $\mathcal{M}(n) =\mathbbm{C}P^n$.

For $\mathcal{M}= S^1$, $\mathcal{M}(n)$ is the total space of the
non-orientable $D^{n - 1}$ disk bundle over $S^1$.

\

The main result of this note is the following:

\begin{theorem}
  Let $\mathcal{M}=\mathbbm{R}P^4 \#\mathbbm{R}P^4 \#\mathbbm{R}P^4$. Then
  $\mathcal{M}$ has the f.p.p. while $\mathcal{M}(2)$ admits a fixed point
  free map. 
\end{theorem}

Here $\#$ stands for the connected sum operation.

\begin{proof}
  Our first observation is about the cohomology ring structure on $H^{\ast}
  (\mathcal{M}; \mathbbm{Z}_2)$. Namely, relatively simple but somewhat tedious
  considerations involving the Mayer-Vietoris exact sequence and the well
  known ring structure of $H^{\ast} (\mathbbm{R}P^4 ; \mathbbm{Z}_2)$ show
  that $H^{\ast} (\mathcal{M}; \mathbbm{Z}_2)$ is (ring) isomorphic with the
  ring
  \[ \mathbbm{Z}_2 [x_1, x_2, x_3] / \langle \{ x_i^5 |i = 1, 2, 3 \}, \{
     x_i^4 + x_j^4 |i \neq j \}, \{ x_i x_j |i \neq j \} \rangle \]
  with $|x_i | = 1$. Given the crucial role of the ring structure on $H^{\ast}
  (\mathcal{M}; \mathbbm{Z}_2)$ in our considerations, we include an appendix
  which contains the necessary computational details.
  
  In particular the cohomology of \ensuremath{\mathcal{M}} has base $\{ 1,
  x_1^n, x_2^n, x_3^n (1 \leqslant n \leqslant 3), x_1^4 \} .$ Also this
  implies that $\chi (\mathcal{M}) = - 1$.
  
  We show $L (f ; \mathbbm{Z}_2) = 1$ for each continuous map $f : \mathcal{M}
  \rightarrow \mathcal{M}$.
  
  To see this let $f^{\ast} \left(\begin{array}{c}
    x_1\\
    x_2\\
    x_3
  \end{array}\right) = A \cdot \left(\begin{array}{c}
    x_1\\
    x_2\\
    x_3
  \end{array}\right)$ where $A$ is a $3 \times 3$ matrix with entries $a_{i
  j}$, $i, j = 1, 2, 3$.
  
  The trace of $f^{\ast}$ is given as follows
  
  \begin{table}[h]
    \begin{tabular}{lll}
      Dimension &  & Trace\\
      0 &  & 1\\
      1 &  & $a_{11} + a_{22} + a_{33}$\\
      2 &  & $a_{11}^2 + a_{22}^2 + a_{33}^2$\\
      3 &  & $a_{11}^3 + a_{22}^3 + a_{33}^3$\\
      4 &  & $a_{11}^4 = a_{22}^4 = a_{33}^4$
    \end{tabular}
    \caption{}
  \end{table}
  
  Now with the $\mathbbm{Z}_2$-coefficients $a_{i j}^2 = a_{i j}$ and hence
  the equation \ $a_{11}^4 = a_{22}^4 = a_{33}^4$ implies $a_{11} = a_{22} =
  a_{33}$. This gives $L (f, \mathbbm{Z}_2) = 1$.
  
  Next we show that $\mathcal{M} (2)$ admits a fixed point free map.
  
  We start with the following observation:
  
  \begin{lemma}
    The Euler characteristic of $\mathcal{M} (2)$ is trivial, i.e. $\chi
    (\mathcal{M} (2)) = 0.$
  \end{lemma}
  
  \begin{flushleft}
  {\tmstrong{Proof of Lemma 1}}: The above claim follows from a very general
  formula cf. [4], Theorem 7.1 on p.145.
  \end{flushleft}
  
  For the completeness of our paper we include a different, shorter and
  self-contained argument. Namely:
  
  The $\mathbbm{Z}_2$-action on $\mathcal{M} \times \mathcal{M}$ is obviously
  smooth and in particular (cf. [10]) simplicial, and hence cellular for some
  CW structure on $\mathcal{M} \times \mathcal{M}$.
  
  Consider the equivariant cellular chain complex $C_{\ast} (\mathcal{M}
  \times \mathcal{M})$. Let $\Delta \subset \mathcal{M} \times \mathcal{M}$ be
  the diagonal, then $\Delta = (\mathcal{M} \times
  \mathcal{M})^{\mathbbm{Z}_2}$ is the fixed point set of the
  $\mathbbm{Z}_2$-action. Thus we have $C_{\ast} (\mathcal{M} \times
  \mathcal{M}) \cong C_{\ast} (\Delta) \oplus \tilde{C} (\mathcal{M} \times
  \mathcal{M})$ where $\widetilde{C_{\ast}} (\mathcal{M} \times \mathcal{M})$
  is an $\mathbbm{Z}_2$-equivariant chain complex generated by cells in
  $\mathcal{M} \times \mathcal{M}$ which are not in $\Delta$. Let $p :
  \mathcal{M} \times \mathcal{M} \rightarrow (\mathcal{M} \times \mathcal{M})
  /\mathbbm{Z}_2 =\mathcal{M} (2)$ be the natural projection on the orbit
  space. Then we have a chain map $p_{\#}$ and the diagram
  
  \begin{table}[h]
    \begin{tabular}{lllll}
      $C_{\ast} (\mathcal{M} \times \mathcal{M})$ & $\cong$ & $C_{\ast}
      (\Delta)$ & $\oplus$ & $\tilde{C} (\mathcal{M} \times \mathcal{M})$\\
      \qquad$\longdownarrow p_{\#}$ &  & \quad$\longdownarrow (p_1)_{\#}$ &  &
      \qquad$\longdownarrow (p_1)_{\#}$\\
      $C_{\ast} (\mathcal{M} (2))$ & $\cong$ & $C_{\ast} (\Delta)$ & $\oplus$
      & $\overline{\tilde{C}} (\mathcal{M} \times \mathcal{M})$
    \end{tabular}
    \caption{}
  \end{table}
  
  Here $\overline{\tilde{C}} (\mathcal{M} \times \mathcal{M})$ is the quotient
  of $\tilde{C}_{\ast} (\mathcal{M} \times \mathcal{M})$ and $(p_1)_{\#}$ are
  corresponding projections.
  
  Now on the chain complex level
  \[ \chi (C_{\ast} (\mathcal{M} \times \mathcal{M})) = \chi (C_{\ast}
     (\Delta)) + \chi (\widetilde{C_{\ast}} (\mathcal{M} \times \mathcal{M}))
  \]
  and analogously
  \[ \chi (C_{\ast} (\mathcal{M} (2))) = \chi (C_{\ast} (\Delta)) + \chi
     (\overline{\tilde{C}}_{\ast} (\mathcal{M} \times \mathcal{M})) \]
  Note that $\chi (\widetilde{C_{\ast}} (\mathcal{M} \times \mathcal{M})) = 2
  \chi (\overline{\tilde{C}}_{\ast} (\mathcal{M} \times \mathcal{M}))$, and
  hence on the level of topological spaces one obtains
  \[ 2 \chi (\mathcal{M} (2)) = \chi (\mathcal{M}) + \chi (\mathcal{M} \times
     \mathcal{M}) = \chi (\mathcal{M}) (1 + \chi (\mathcal{M})) = 0 \]
  and hence $\chi (\mathcal{M} (2)) = 0$ as claimed.
  
  Finally the symmetric square $\mathcal{M} (2)$ is obviously a simplicial
  complex of dimension 8.
  
  It is a rational homology manifold (cf. [2]). In particular it means that
  for each vertex $v \in \mathcal{M} (2)$ the link $\tmop{Ln} (v) = \partial |
  \tmop{St} (v) |$ has the rational homology of $S^7$, here $\tmop{St} (v)$ is
  the star of $v$. This implies that $\mathcal{M} (2)$ is a polyhedron of type
  $\mathcal{W}$ in the sense of [5] p.143, with $\chi (\mathcal{M} (2)) = 0$.
  
  Consequently by the Theorem 1 (the converse of the Lefschetz Deformation
  Theorem) in [5] p.143, $\mathcal{M} (2)$ admits a fixed point free
  deformation.
\end{proof}

\begin{flushleft}
{\LARGE \textbf{Remarks and comments}}
\end{flushleft}

The example of closed manifolds $\mathcal{M}, \mathcal{N}$ with the f.p.p. for
which $\mathcal{M} \times \mathcal{N}$ fails to have the f.p.p. presented in
[11] is surprisingly complicated. One attempt to construct ``simple'' examples
of this sort could be to consider products of basic manifolds with the f.p.p..

These basic examples are: $\mathbbm{R}P^{2 n}, \mathbbm{C}P^{2 n}, n = 1, 2,
\cdots$ and $\mathbbm{H}P^n, n = 2, 3, 4, \cdots$. i.e., the corresponding
real, complex and quaternionic projective spaces.

It turns out that mixing different projective spaces, i.e., forming

(a)$\mathbbm{R}P^{2 m} \times \mathbbm{C}P^{2 n}$

(b)$\mathbbm{R}P^{2 m} \times \mathbbm{H}P^n$

(c)$\mathbbm{C}P^{2 m} \times \mathbbm{H}P^n$

one ends up with manifold with the f.p.p. cf. [9], Theorem 4.7.

It appears that a more involved argument would show that the Cartesian powers
of these manifolds have the f.p.p..

The case of $\mathbbm{R}P^{2 n}$ is simple (use the Lefschetz Fixed Point
Theorem with the rational coefficients). The considerations involving
$\mathbbm{C}P^{2 n}$ and $\mathbbm{H}P^n$ are more involved. As an example we
check the following crucial case.

\begin{theorem}
  The Cartesian power $(\mathbbm{C}P^2)^n =\mathbbm{C}P^2 \times
  \mathbbm{C}P^2 \times \cdots \times \mathbbm{C}P^2$ has the f.p.p..
\end{theorem}

\begin{proof}
  Let $f : (\mathbbm{C}P^2)^n \rightarrow (\mathbbm{C}P^2)^n$ be a map. We
  show that the Lefschetz number computed with the $\mathbbm{Z}_2$-coefficient
  is given by $L (f ; \mathbbm{Z}_2) = 1$.
  
  Consider the induced homomorphism
  \[ f^{\ast} : H^{\ast} ((\mathbbm{C}P^2)^n) \longrightarrow H^{\ast}
     ((\mathbbm{C}P^2)^n) \]
  By the Kunneth Formula, $H^{\ast} ((\mathbbm{C}P^2)^n)$ can be identified
  with the $n$-fold tensor product $H^{\ast} (\mathbbm{C}P^2) \otimes \cdots
  \otimes H^{\ast} (\mathbbm{C}P^2)$. Let $X_i, 1 \leqslant i \leqslant n$ be
  the generator of $H^2 ((\mathbbm{C}P^2)^n)$ corresponding to $1 \otimes
  \cdots \otimes 1 \otimes x \otimes 1 \otimes \cdots \otimes 1$, where $x$ in
  the $i$th place is a fixed generator of $H^2 (\mathbbm{C}P^2)$.
  
  Assume that $f^{\ast} \left(\begin{array}{c}
    x_1\\
    \vdots\\
    x_n
  \end{array}\right) = A \cdot \left(\begin{array}{c}
    x_1\\
    \vdots\\
    x_n
  \end{array}\right)$ for a matrix $A$ given by $A = \{ a_{i j} \}$ $1
  \leqslant i, j \leqslant n$. Let $X_{k, l} = \{ x_{i_1}^2 x_{i_2}^2 \cdots
  x_{i_k}^2 x_{j_1} \ldots x_{j_l} |i_1, \cdots, i_k, j_1, \cdots, j_k
  \tmop{are} k + l \tmop{distinct} \tmop{integers} \tmop{between} 1 \tmop{and}
  n \}$. To be more precise, $\{ i_1, \cdots, i_k \}, \{ j_1, \cdots, j_k \}$
  go through all mutually distinct $k, l$ subsets of $\{ 1, \cdots, n \}$.
  
  Then $X = \underset{l \geqslant 0}{\underset{k \geqslant 0}{\underset{k + l
  \leqslant n}{\cup}}} X_{k, l}$ is a basis for $H^{\ast}
  ((\mathbbm{C}P^2)^n)$, where $X_{0, 0}$ is the basis for $H^0
  ((\mathbbm{C}P^2)^n) =\mathbbm{Z}_2$. Now $L (f)$ is the trace of $f^{\ast}$
  with respect to $X$.
  
  Let $T_{k, l}$ be the trace of $f^{\ast}$ generated by $X_{k, l}$. Then we
  claim the following:
  
  (1)$T_{0, 0} = 1$
  
  (2)$T_{k, l} = T_{l, k}$
  
  (3)$T_{k, k} = 0$ for $k \geqslant 1$
  
  Note that these claims imply $L (f ; \mathbbm{Z}_2) = 1$ completing the
  proof of Theorem 3. With respect to the proof of (1), (2) and (3):\\
  
  \begin{flushleft}
  The claim (1) is obvious.
  \end{flushleft}
  
  \begin{flushleft}
  {\underline{\textbf{Proof of the claim (2)}}}:
  \end{flushleft}
  
  Let $t_{x_{i_1}^2 \cdots x_{i_k}^2 x_{j_1} \ldots x_{j_l}}$ be the trace
  generated by the element $x_{i_1}^2 \cdots x_{i_k}^2 x_{j_1} \ldots
  x_{j_l}$. It suffices to show that $t_{x_{i_1}^2 \cdots x_{i_k}^2 x_{j_1}
  \ldots x_{j_l}} = t_{x_{j_1}^2 \cdots x_{j_l}^2 x_{i_1} \ldots x_{i_k}}$,
  for any distinct $i_1, \cdots, i_k, j_1, \cdots, j_l$.
  
  We have
  
  $f^{\ast} (x_{i_s}) = \underset{r = 1}{\overset{n}{\Sigma}} a_{i_s r} x_r,
  1 \leqslant s \leqslant k$,
  
  $f^{\ast} (x_{j_t}) = \underset{r = 1}{\overset{n}{\Sigma}} a_{j_t r} x_r, 1
  \leqslant t \leqslant l$.
  
  Thus $f^{\ast} (x_{i_s}^2) = f^{\ast} (x_{i_s})^2 = \underset{r =
  1}{\overset{n}{\Sigma}} a_{i_s r}^2 x_r^2 = \underset{r =
  1}{\overset{n}{\Sigma}} a_{i_s r} x_r^2$. Similarly $f^{\ast} (x_{j_t}^2) =
  \underset{r = 1}{\overset{n}{\Sigma}} a_{j_t r} x_r^2$.
  
  So
  \[ f^{\ast} (x_{i_1}^2 \cdots x_{i_k}^2 x_{j_1} \ldots x_{j_l}) = \left(
     \underset{s = 1}{\overset{k}{\Pi}} \underset{r = 1}{\overset{n}{\Sigma}}
     a_{i_s r} x_r^2 \right) \cdot \left( \underset{t = 1}{\overset{l}{\Pi}}
     \underset{r = 1}{\overset{n}{\Sigma}} a_{j_t r} x_r \right) \]
  and analogously
  \[ f^{\ast} (x_{j_1}^2 \cdots x_{j_l}^2 x_{i_1} \ldots x_{i_k}) = \left(
     \underset{t = 1}{\overset{l}{\Pi}} \underset{r = 1}{\overset{n}{\Sigma}}
     a_{j_t r} x_r^2 \right) \cdot \left( \underset{s = 1}{\overset{k}{\Pi}}
     \underset{r = 1}{\overset{n}{\Sigma}} a_{i_s r} x_r \right) \]
  From this it is not difficult to see that
  \[ t_{x_{i_1}^2 \cdots x_{i_k}^2 x_{j_1} \ldots x_{j_l}} = t_{x_{j_1}^2
     \cdots x_{j_l}^2 x_{i_1} \ldots x_{i_k}} \]
  Namely, both of them are given by
  \[ \left|\begin{array}{ccc}
       a_{i_1 i_1} & \cdots & a_{i_1 i_k}\\
       \vdots &  & \vdots\\
       a_{i_k i_1} & \cdots & a_{i_k i_k}
     \end{array}\right| \cdot \left|\begin{array}{ccc}
       a_{j_1 j_1} & \cdots & a_{j_1 j_l}\\
       \vdots &  & \vdots\\
       a_{j_l j_1} & \cdots & a_{j_l j_l}
     \end{array}\right| \]

  \begin{flushleft}
  {\underline{\textbf{Proof of the claim (3)}}}: 
  \end{flushleft}
 
  We have
  \[ T_{k, k} = \underset{j_1, \cdots, j_k}{\underset{i_1, \cdots,
     i_k}{\Sigma}} t_{x_{i_1}^2 \cdots x_{i_k}^2 x_{j_1} \ldots x_{j_k}} \]
  But $t_{x_{i_1}^2 \cdots x_{i_k}^2 x_{j_1} \ldots x_{j_k}} + t_{x_{j_1}^2
  \cdots x_{j_k}^2 x_{i_1} \ldots x_{i_k}} = 2 t_{x_{i_1}^2 \cdots x_{i_k}^2
  x_{j_1} \ldots x_{j_k}} = 0$.
\end{proof}

\

\begin{flushleft}
{\tmstrong{{\LARGE Appendix}}}
\end{flushleft}

\begin{theorem}
  The cohomology ring $H^{\ast} \left( \underset{i = 1}{\overset{n}{\#}} P^{2
  k} ; \mathbbm{Z}_2 \right)$, $k \geqslant 2$ is isomorphic to $\mathbbm{Z}_2
  [x_1, \cdots, x_n] / \langle x_1^{2 k + 1}, \{ x_i^{2 k} + x_j^{2 k} |i \neq
  j \}, \{ x_i x_j |i \neq j \} \rangle$, $| x_i | = 1$
\end{theorem}

\begin{proof}
  We shall omit the coefficients since it will always be $\mathbbm{Z}_2$.
  
  The additive structure comes easily from the integral homology and Universal
  Coefficient Theorem. We only need to determine the multiplicative structure.
  
  To do this we will proceed by induction.
  
  For inductive purpose we shall prove a stronger version of the above
  theorem.
  
  Denote $\underset{i = 1}{\overset{n}{\#}} P^{2 k}$ by $P_n$, treated as
  $S^{2 k} \underset{i = 1}{\overset{n}{\#}} P^{2 k}$ where all disks cut from
  $S^{2 k}$ have positive distance between each other. For $n = 1$, write $P_1
  = P^{2 k}$ as $P$.
  
  Define a map $p_i^n : P_n \rightarrow P$ by fixing the $i$th copy of $P$
  (with the open disk removed) in $P_n$, mapping an ``annulus'' in $S^{2 k}$
  near the boundary of this disk via radial projection onto the open disk in
  $P$ and sending the remainder onto the center of that disk. Let $x$ be the
  generator of $H^1 (P ; \mathbbm{Z}_2)$. We claim in addition that in the
  Theorem 7, $x_i$ can be chosen as $p_i^{n \ast} (x)$.
  
  The case $n = 1$ is well-known.
  
  For any $n$, let $\overline{P_n}$ be $P_n$ with yet another open disk
  (disjoint with the existing ones) removed from $S^{2 k}$. Denote
  $\overline{P_1}$ as $\bar{P}$. Note that $\bar{P}$ is $P^{2 k}$ with an open
  disk removed.
  
  Now assume that the stronger version of the above theorem holds for $n$
  copies of $P^{2 k}$, i.e., for $P_n$. We shall prove it for $P_{n + 1}$.
  
  By definition, $P_{n + 1} = \overline{P_n} \cup \bar{P}$, $\overline{P_n}
  \cap \bar{P} = S^{2 k - 1}$, where $\overline{P_n}$ corresponds to the first
  $n$ copies of $P$ in $P_{n + 1}$.
  
  From the Mayor-Vietoris Sequence of $P_n = \overline{P_n} \cup D^{2 k}$ and
  using the fact that $H^{2 k} (\overline{P_n}) = 0$ (this is because
  $\overline{P_n}$ is homotopy equivalent to a non-compact $2 k$-manifold),
  one can see that the inclusion $\overline{P_n} \hookrightarrow P_n$ induces
  isomorphisms on $H^m$ for $0 \leqslant m \leqslant 2 k - 1$ and for any $n$.
  
  An argument by M-V sequence with respect to $P_{n + 1} = \overline{P_n} \cup
  \bar{P}$ \ similar to the one above shows that
  \[ H^m (P_{n + 1}) \overset{i_n^{\ast} \oplus i_1^{\ast}}{\longrightarrow}
     H^m (\overline{P_n}) \oplus H^m (\bar{P}) \]
  is an isomorphism for $0 \leqslant m \leqslant 2 k - 1$, where $i_1, i_n$
  are canonical inclusions.
  
  There is a projection $q_n : P_{n + 1} \rightarrow P_n$ (defined similarly
  as $p_i^n$ above) that is identity on $\overline{P_n}$ and maps $\bar{P}$
  onto the disk $D^{2 k}$. It is not hard to show $p_i^{n + 1} = p_i^n \circ
  q_n, 1 \leqslant i \leqslant n$.
  
  Now consider the composition;
  \[ H^m (P_n) \oplus H^m (P)  \overset{q_n^{\ast} \oplus p_{n + 1}^{n + 1
     \ast}}{\longrightarrow} H^m (P_{n + 1}) \overset{i_n^{\ast} \oplus
     i_1^{\ast}}{\longrightarrow} H^m (\overline{P_n}) \oplus H^m (\bar{P}) \]
  for $1 \leqslant m \leqslant 2 k - 1$.
  
  We have proven that the inclusions $i_n \circ q_n$ and $i_1 \circ p_{n +
  1}^{n + 1}$ induce isomorphism on $H^m$. On the other hand, $i_n \circ p_{n
  + 1}^{n + 1}$ and $i_1 \circ q_n$ are null-homotopic, whence $(q_n^{\ast}
  \oplus p_{n + 1}^{n + 1 \ast}) \circ (i_n^{\ast} \oplus i_1^{\ast}) =
  (q_n^{\ast} \circ i_n^{\ast}) \oplus (p_{n + 1}^{n + 1 \ast} \circ
  i_1^{\ast})$ is an isomorphism. We have seen that $i_n^{\ast} \oplus
  i_1^{\ast}$ is an isomorphism, thus the same is true for $q_n^{\ast} \oplus
  p^{n + 1\ast}_{n + 1}$.
  
  Now define $x_i = p_i^{n + 1 \ast} (x) \in H^1 (P_{n + 1}), 1 \leqslant i
  \leqslant n + 1$, then $x_i = q_n^{\ast} \circ p_i^{n \ast} (x)$ for $1
  \leqslant i \leqslant n$.
  
  The inductive assumption implies that for $1 \leqslant m \leqslant 2 k - 1$,
  $\{ p_i^{n \ast} (x)^m = p_i^{n \ast} (x^m), 1 \leqslant i \leqslant n \}$
  is a basis for $H^m (P_n)$.
  
  Since $q_n^{\ast} \oplus p^{n + 1\ast}_{n + 1}$ is an isomorphism, $H^m
  (P_{n + 1})$ has basis $\{ x_1^m, \cdots, x_n^m, x_{n + 1}^m \}$, $1
  \leqslant m \leqslant 2 k - 1$.
  
  Next we turn to dimension $2 k$.
  \\
  
  {\tmstrong{{\underline{Claim}}:}} Both $q_n$ and $p_{n + 1}^{n + 1}$ induce
  isomorphism on $H^{2 k}$.
  \\
  
  {\underline{{\tmstrong{Proof of the claim}}}}: Consider the commutative
  diagram:
  
  \begin{table}[h]
    \begin{tabular}{lll}
      $P_{n + 1}$ & $\overset{p_{n + 1}^{n + 1}}{\longrightarrow}$ &
      \quad$P$\\
      \quad$\longdownarrow$ &  & \quad$\longdownarrow$\\
      $(P_{n + 1}, \overline{P_n})$ & $\overset{p_{n + 1}^{n +
      1}}{\longrightarrow}$ & $(P, D^{2 k})$\\
      \quad$\longdownarrow$ &  & \quad$\longdownarrow$\\
      $(P_{n + 1} / \overline{P_n}, \ast)$ &
      $\overset{\tilde{p}}{\longrightarrow}$ & $(P / D^{2 k}, \ast)$
    \end{tabular}
    \caption{}
  \end{table}
  
  where $\tilde{p}$ is induced by $p_{n + 1}^{n + 1}$ and the vertical maps
  are canonical inclusions or projections. $\tilde{p}$ is a homeomorphism and
  the two lower vertical maps induce isomorphism on cohomology. Consider the
  long exact sequence
  \[ \cdots \longrightarrow H^{2 k + 1} (\overline{P_n}) \longrightarrow H^{2
     k} (P_{n + 1}, \overline{P_n}) \longrightarrow H^{2 k} (P_{n + 1})
     \longrightarrow H^{2 k} (\overline{P_n}) \longrightarrow \cdots \]
  Since $H^{2 k + 1} (\overline{P_n}) = 0 = H^{2 k} (\overline{P_n})$, the
  upper left map in the above diagram induces isomorphism on cohomology.
  Trivially, $P \rightarrow (P, D^{2 k})$ induces isomorphism on $H^{2 k}$.
  
  Combining the above arguments and using commutativity, we have shown that
  $p_{n + 1}^{n + 1} : P_{n + 1} \rightarrow P$ induces an isomorphism on
  $H^{2 k}$.
  
  In much the same way one can show that $q_n$ induces isomorphism on $H^{2
  k}$. This finishes the proof of the claim.
  
  \
  
  The claim together with the inductive assumption implies that $H^{2 k}
  (P_{n + 1})$ is generated by $x_1^{2 k} = x_2^{2 k} = \cdots = x_n^{2 k} =
  x_{n + 1}^{2 k}$.
  
  It remains to show that $x_i x_j = 0, i \neq j, 1 \leqslant i, j \leqslant n
  + 1$.
  
  For the case $n = 1$, let $x_1 x_2 = a x_1^2 + b x_2^2$ (this is because $\{
  x_1^2, x_2^2 \}$ are basis) for some $a, b$. Since one can exchange the role
  of $x_1$ and $x_2$ (by exchanging the two copies of $P$ in $P_2$), we must
  have $a = b$.
  
  Suppose $a = b = 1$, then $x_1^2 x_2 = x_1 (x_1 x_2) = x_1^3 + x_1 x_2^2$,
  whence $x_1^3 = x_1^2 x_2 + x_1 x_2^2$. Similarly $x_2^3 = x_1^2 x_2 + x_1
  x_2^2$. This contradicts to $\{ x_1^3, x_2^3 \}$ being basis.
  
  Consequently $a = b = 0$ and the claim is proven for $n = 1$.
  
  For the case $n > 1$, we decompose $p_i^{n + 1}$, $p_j^{n + 1}$ into
  commutative diagrams:
  
  \
  
  \tmfloat{h}{small}{table}{\begin{tabular}{lll}
    $P_{n + 1}$ &  & \\
    & $\overset{p_{i j}}{\searrow}$ & \\
    $p_i^{n + 1} \longdownarrow$ &  & $P_2$\\
    & $\underset{p_i'}{\swarrow}$ & \\
    $P$ &  & 
  \end{tabular}}{}\tmfloat{h}{small}{table}{\begin{tabular}{lll}
    $P_{n + 1}$ &  & \\
    & $\overset{p_{i j}}{\searrow}$ & \\
    $p_j^{n + 1} \longdownarrow$ &  & $P_2$\\
    & $\underset{p_j'}{\swarrow}$ & \\
    $P$ &  & 
  \end{tabular}}{}
  
  Here $p_{i j}$ preserves the $i$th and $j$th copy of $P$ in $P_n$ while
  project other copies of $P$ onto disks, and $p_i'$ (resp. $p_j'$) preserves
  the $i$th (resp. $j$th) copy of $P$ while projects the other onto respective
  disks.
  
  Then $x_i x_j = p_i^{n + 1 \ast} (x) p_j^{n + 1 \ast} (x) = p_{i j}^{\ast}
  [p_i^{\prime \ast} (x) \cdot p_j^{\prime \ast} (x)] = p_{i j}^{\ast} (0) =
  0$ by the inductive assumption.
  
  This finishes both the inductive step and the proof of Theorem 4.
\end{proof}

\

\begin{flushleft}
{\tmstrong{{\Large References}}}:
\\\

[1] P. Blagojevi{\'c}, V. Gruji{\'c}, R. {\v Z}ivaljevi{\'c}, {\tmem{Symmetric
products of surfaces: A unifying theme for topology and physics}}. Summer
School in Modern Mathematical Physics, SFIN, XV (A3), Institute of Physics,
Belgrade (2002).

[2] A. Borel, {\tmem{Seminar on transformation groups}}, Annals of Mathematics
Studies No. 46, Princeton, 1960.

[3] R. Bott,{\tmem{ On symmetric products}}~and the Steenrod squares, Ann of
Math. (2) 57 (1953), 579--590.

[4] G. Bredon, {\tmem{Introduction to Compact Transformation Groups}}. (Pure
and Applied Mathematics, Vol. 46)

[5] R. Brown,{\tmem{ The Lefschetz}}~Fixed Point Theorem, Scott, Foresman,
Glenview, IL, 1971

[6] R. Brown, {\tmem{The Fixed Point Property and Cartesian Products}}, Amer.
Math. Month. 89, 654--678.

[7] A. Dold, Homotopy of symmetric products and other functors of complexes,
Ann. Math. 68(1958), 54-80.

[8] A. Dold, R. Thom, {\tmem{Quasifaserungen und unendliche symmetrische
produkte}}, Ann. of Math. (2) 67 (1958), 239--281.

[9] E. Fadell, {\tmem{Recent}} {\tmem{results in the fixed point theory of
continuous maps}}. Bull. Amer. Math. Soc. 76 (1970), no. 1, 10--29.

[10] E. Fadell, {\tmem{Some examples in fixed point theory}}. Pacific J. Math.
33 (1970), no. 1, 89--100.

[11] S. Husseini, {\tmem{The product of manifolds with the f.p.p. need not
have the f.p.p.}}, Amer. J. Math. 99 (1977) 919--931.

[12] S. Illman. {\tmem{Smooth equivariant triangulations of G-manifolds for G
a finite group}}. Math. Ann., 233(3):199--220, 1978.

[13] K. Kuratowski, {\tmem{Problem 49}}, Fund. Math. 15 (1930), 356.

[14] W. Lopez, {\tmem{An example in the fixed point theory of polyhedra}},
Bull. Amer. Math. Soc. 73 (1967), 922-924.

[15] H. R. Morton, {\tmem{Symmetric products of the circle}}, Proc. Camb.
Phil. Soc. 63 (1967), pp. 349-352.

\

Slawomir Kwasik

Department of Mathematics

Tulane University

New Orleans, LA, 70118

kwasik@tulane.edu

\

Fang Sun

Department of Mathematics

Tulane University

New Orleans, LA, 70118

fsun@tulane.edu
\end{flushleft}

\end{document}